\documentclass[10pt]{article}

\usepackage{authblk}
\usepackage{url}
\usepackage{tikz}
\usepackage{pdflscape}
\usepackage{longtable}
\usepackage{booktabs}
\usepackage{colortbl}
\usepackage{verbatim}
\usepackage{amssymb}
\usepackage{amsmath}
\usepackage{bm}
\usepackage{amsthm, color, fullpage}
\usepackage{etoolbox}

\newcommand{\C}{\mathbb{C}}

\newcommand{\Z}{\mathbb{Z}}

\newcommand{\R}{\mathbb{R}}

\DeclareMathOperator{\curl}{curl}
\DeclareMathOperator{\divv}{div}

\newtheorem{thm}{Theorem}[section]
\newtheorem{cor}[thm]{Corollary}

\newtheorem{lem}[thm]{Lemma}
\newtheorem{prop}[thm]{Proposition}
\newtheorem{defi}[thm]{Definition}
\newtheorem{rmk}[thm]{Remark}

\makeatletter
\DeclareRobustCommand\widecheck[1]{{\mathpalette\@widecheck{#1}}}
\def\@widecheck#1#2{%
    \setbox\z@\hbox{\m@th$#1#2$}%
    \setbox\tw@\hbox{\m@th$#1%
       \widehat{%
          \vrule\@width\z@\@height\ht\z@
          \vrule\@height\z@\@width\wd\z@}$}%
    \dp\tw@-\ht\z@
    \@tempdima\ht\z@ \advance\@tempdima2\ht\tw@ \divide\@tempdima\thr@@
    \setbox\tw@\hbox{%
       \raise\@tempdima\hbox{\scalebox{1}[-1]{\lower\@tempdima\box
\tw@}}}%
    {\ooalign{\box\tw@ \cr \box\z@}}}
\makeatother

\begin{document}

\title{\sc Spatial non-locality 
	of the Maxwell system on periodic structures}

\author[1]{Kirill Cherednichenko}
\author[1]{Serena D'Onofrio}
\affil[1]{Department of Mathematical Sciences, University of Bath, Claverton Down, Bath, BA2 7AY, United Kingdom}

\maketitle

\begin{abstract}
\noindent For 
$\varepsilon>0,$ we 
analyse the Maxwell system of equations of electromagnetism on $\varepsilon$-periodic sets
$S^\varepsilon\subset{\mathbb R}^3.$ Assuming that a family of Borel measures $\mu^\varepsilon,$ such that 
${\rm supp}(\mu^\varepsilon)=S^\varepsilon,$ is obtained by $\varepsilon$-contraction of a fixed periodic measure $\mu,$ and for right-hand sides $f^\varepsilon\in L^2(\R^3, d\mu^\varepsilon),$ we prove order-sharp norm-resolvent convergence estimates for the solutions of the system. 

\vskip 0.3cm


\end{abstract}


\section{Introduction}\label{sec1}
The aim of this work is to obtain norm-resolvent homogenisation estimates for the stationary system of Maxwell equations of electromagnetism. We are interested in the geometric setting of singular periodic structures described by arbitrary periodic Borel measures. In our earlier work \cite{CDO22}, we derived such estimates for a family of problems describing the magnetic field $H^\varepsilon:$ 
\begin{equation}\label{eq_paper2}
\curl A(\cdot /\varepsilon) \curl H^\varepsilon+H^\varepsilon=f, \qquad f\in\bigl[L^2(\R^3,d\mu^\varepsilon)\bigr]^3,\quad {\rm div}f=0,\qquad \varepsilon>0,
\end{equation}
where the $\varepsilon$-periodic 
measure $\mu^\varepsilon$ is an $\varepsilon$-rescaling of a $Q$-periodic ($Q:=[0,1)^3$) measure, and $A$ is a symmetric, bounded, and uniformly positive definite $Q$-periodic matrix-valued function.
 The problem \eqref{eq_paper2}, understood in the sense of an appropriate integral identity, is the resolvent form of the Maxwell system of equations in electromagnetism in the absence of external currents. For each $\varepsilon,$ the function $H^\varepsilon$ represents the divergence-free magnetic field, 
the coefficient matrix $A$ is the inverse of the dielectric permittivity relative to vacuum, and the medium is assumed to be non-magnetic (i.e., its magnetic permeability is assumed to be equal to that of vacuum.) The vector function $f$ on the right-hand side is a fictitious ``magnetic charge density'', which is a mathematical artefact of the operator resolvent formulation.

Our approach \cite{CDO22} to the analysis of the asymptotic behaviour of solutions to (\ref{eq_paper2}) was based on the study of the family of operators, parametrised by the quasimomentum $\theta$, obtained from \eqref{eq_paper2} by an appropriate generalisation of the classical Floquet transform accounting for the arbitrary choice of the measure $\mu.$ The related strategy consisted in constructing an asymptotic approximation in powers of $\varepsilon$ for each value of $\theta,$ analysing the homogenisation corrector as a function of $\varepsilon$ and $\theta$, and obtaining an estimate uniform with respect to $\theta$ for the remainder. The principal tool for the latter was a $\theta$-uniform Poincar\'{e}-type inequality in appropriate Sobolev spaces of quasiperiodic functions associated with periodic Borel measures. The result of \cite{CDO22} allowed us to estimate the magnetic field $H^\varepsilon$ and the magnetic induction directly from the solution of the related homogenised equation. Different estimates were proved for the electric field and electric displacement, where the leading order of the approximation contains finite rapidly oscillating terms.


Consider a $Q$-periodic Borel measure $\mu$ on $\R^3$
such that $\mu(Q)=1$. For each $\varepsilon>0,$ we define the ``$\varepsilon$-scaling" of $\mu,$ i.e. the $\varepsilon$-periodic measure $\mu^\varepsilon$ given by 
$\mu^\varepsilon({B})=\varepsilon^3 \mu(\varepsilon^{-1}{B})$ for all Borel sets ${B}\subset \R^3,$ so that $\mu^1\equiv\mu.$
Henceforth, we denote by $C_0^\infty({\mathbb R}^3)$ the set of of infinitely smooth functions with compact support in ${\mathbb R}^3$ and by $L^2({\mathbb R}^3, d\mu^\varepsilon)$ the space of functions with values in ${\mathbb C}^3$ that are square integrable over ${\mathbb R}^3$ with respect to the measure $\mu^\varepsilon$ --- this is the same as the space $[L^2(\R^3,d\mu^\varepsilon)]^3$ in (\ref{eq_paper2}). Throughout the paper, for vectors $a, b\in{\mathbb C}^2$ we denote by $a\cdot b$ their standard (sesquilinear) Euclidean inner product, and define all function spaces over the filed ${\mathbb C}.$

 Our aim here is to analyse the asymptotic behaviour, as $\varepsilon\to 0$, of solution pairs $(D^\varepsilon, B^\varepsilon)$ for the system of Maxwell equations ({\it cf.} \cite{Cessenat})
\begin{equation}
\label{Max_general}
	\curl (A (\cdot/\varepsilon) D^\varepsilon)+B^\varepsilon=0,
\qquad \curl(\widetilde{A}(\cdot/\varepsilon)B^\varepsilon)-D^\varepsilon=J.
\end{equation}
Here, the divergence-free vector function $J\in L^2(\R^3,d\mu^\varepsilon)$ represents the current density, and the matrix-valued functions $A,$ $\widetilde{A}$ which are 
	assumed to be $Q$-periodic, symmetric, positive definite, and continuously differentiable,
represent the inverse of the relative dielectric permittivity and the inverse relative magnetic permeability, respectively, see \cite[Appendix A]{CDO22} for non-dimensionalising the Maxwell system. 

For each $\varepsilon>0,$ the system (\ref{Max_general}) is understood in the variational sense: one looks for $(D^\varepsilon,B^\varepsilon)\in{\mathcal X}\oplus{\mathcal X},$ where ${\mathcal X}$ in an appropriate space of divergence-free functions, such that
\begin{equation}
\int_{{\mathbb R}^3}\bigl(A (\cdot/\varepsilon) D^\varepsilon\cdot\curl\varphi_1+B^\varepsilon\cdot\varphi_2\bigr)d\mu^\varepsilon\\[0.25em]
+\int_{{\mathbb R}^3}\bigl(B^\varepsilon\cdot\curl\varphi_2-D^\varepsilon\cdot\varphi_1\bigr)d\mu^\varepsilon=\int_{\mathbb R}J\cdot\varphi_2\,d\mu^\varepsilon,
\label{Var_form}
\end{equation}
for all vectors $(\varphi_1, \varphi_2)$
each of whose components is an element of a dense subset of ${\mathcal X}.$ 
The vector $(B^\varepsilon, D^\varepsilon)$ in (\ref{Var_form}) represents the (time-independent) amplitude of a time-harmonic field for which the variation of the action functional\footnote{Here the action functional is an integral whose spatial part is expressed in terms of the measure $d\mu^\varepsilon.$} for the time-dependent system of Maxwell equations  \cite{Cessenat}, \cite{Jackson} vanishes.

Our goal is to obtain norm-resolvent estimates for the difference between the solution to \eqref{Max_general} for small $\varepsilon$ and 
	the solution to a suitable ``homogenised" problem that replaces the formally suggested system 
\begin{equation}
	\label{Max_general_homo}
		\curl\bigl(A^{\rm hom} {D}^{\rm hom}\bigr)+{B}^{\rm hom}=0,\qquad 
		\curl(\widetilde{A}^{\rm hom}{B}^{\rm hom})-{D}^{\rm hom}=J,
\end{equation} 
where 
	$A^{\rm hom},$ $\widetilde{A}^{\rm hom}$ are matrices representing the inverse of the ``homogenised" electric permittivity and magnetic permeability.

While the ``limit" system (\ref{Max_general_homo}) is provided by the standard two-scale asymptotic expansion, it turns out to be an incorrect effective model if one were to require norm-resolvent (or even strong) convergence as $\varepsilon\to0.$ 
We will show that the correct replacement of (\ref{Max_general_homo}) involves an $\varepsilon$-dependent pseudodifferential operator, which can be viewed as a singular perturbation of  (\ref{Max_general_homo}).



Our first step in tackling the system \eqref{Max_general} is to rewrite it in an equivalent symmetric form, for
which we follow the approach of \cite[Section 3]{Suslina2005}.
Labelling $\sqrt{A(\cdot/\varepsilon)}D^\varepsilon := \mathfrak{D}^\varepsilon,$ we write, for each $\varepsilon>0,$
\begin{equation}
	\label{Max_general_g}
\sqrt{A(\cdot/\varepsilon)}\curl\bigl\{\widetilde{A}\curl\bigl(\sqrt{A(\cdot/\varepsilon)} \mathfrak{D}^\varepsilon\bigr)\bigr\}+ \mathfrak{D}^\varepsilon=-\sqrt{A(\cdot/\varepsilon)}J, 
\end{equation}
Denote by  $C^1_0(\R^3)$ the set of continuously differentiable vector functions with compact support in $\R^3$ and define the space $H^1_{\curl A^{1/2}}(\R^3, d\mu^\varepsilon)$ as the closure of the set of pairs 
$\bigl\{\bigl(\phi, \curl (\sqrt{A(\cdot/\varepsilon)}\phi)\bigr):\  \phi\in [C^1_0(\R^3)]^3\bigr\}$
in the direct sum $L^2(\R^3,d\mu^\varepsilon)\oplus L^2(\R^3,d\mu^\varepsilon)$. 
We say that 
$\bigl(\mathfrak{D}^\varepsilon, \curl\bigl(\sqrt{A(\cdot/\varepsilon)} \mathfrak{D}^\varepsilon\bigr)\bigr)\in H^1_{\curl A^{1/2}}(\R^3, d\mu^\varepsilon)$
is a solution of \eqref{Max_general_g} if
\begin{equation}
\label{Max_general_g_weak}
\begin{aligned}
\int_{\R^3}\widetilde{A}(\cdot/\varepsilon)\curl\bigl(\sqrt{A(\cdot/\varepsilon)} \mathfrak{D}^\varepsilon\bigr)\cdot{\curl \bigl(\sqrt{A(\cdot/\varepsilon)}\varphi\bigr)}+\int_{\R^3} \mathfrak{D}^\varepsilon\cdot{\varphi}&=-\int_{\R^3}\!\sqrt{A(\cdot/\varepsilon)}J \cdot{\varphi}\quad
\forall \varphi\in \bigl[C_0^\infty({\mathbb R}^3)\bigr]^3.
\end{aligned}
\end{equation}
Clearly, the set of test functions in the identity (\ref{Max_general_g_weak}) can be equivalently replaced by $H^1_{\curl A^{1/2}}(\R^3, d\mu^\varepsilon).$ 

For every $\varepsilon>0,$ the left-hand side of \eqref{Max_general_g_weak} defines an inner product in $H^1_{\curl A^{1/2}}(\R^3,d\mu^\varepsilon)$. The right-hand side is linear bounded functional on $H^1_{\curl A^{1/2}}(\R^3,d\mu^\varepsilon)$ with respect to this inner product, hence the existence and uniqueness of the solution of \eqref{Max_general_g_weak} are a consequence of the Riesz representation theorem. In what follows we study the resolvent of the operator $\mathcal{A}^\varepsilon$ with domain
\begin{equation}
	\begin{aligned}
&{\rm dom} (\mathcal{A}^\varepsilon)=\biggl\{ u \in L^2(\R^3,d\mu^\varepsilon): \; \exists \curl\bigl(\sqrt{A(\cdot/\varepsilon)}u\bigr)\ {\rm such\ that}\\[0.4em]
&\int_{\R^3}\widetilde{A}(\cdot/\varepsilon)\curl\bigl(\sqrt{A(\cdot/\varepsilon)}u\bigr)\cdot{\curl\bigl(\sqrt{A(\cdot/\varepsilon)} \varphi\bigr)}+\int_{\R^3} u\cdot{\varphi}=-\int_{\R^3}\sqrt{A(\cdot/\varepsilon)}g \cdot{\varphi} \qquad \forall \varphi\in \bigl[C_0^\infty({\mathbb R}^3)\bigr]^3
\\[0.3em]
& \qquad \qquad\qquad\qquad{\rm for \; some} \; g\in L^2(\R^3, d\mu^\varepsilon), \; \divv g=0\biggr\}
\end{aligned}
	\label{dom_A_epsilon_cur}
\end{equation}
and defined by the formula
$\mathcal{A}^\varepsilon u= -\sqrt{A(\cdot/\varepsilon)}g - u,$
where the divergence-free $g\in L^2(\R^3, d\mu^\varepsilon)$
 and $u\in \rm{dom}(\mathcal{A}^\varepsilon)$ are linked\footnote{It is not difficult to show that for each $u\in {\rm dom}({\mathcal A}^\varepsilon)$ there exists only one $g$ with the property described in (\ref{dom_A_epsilon_cur}).}  as in (\ref{dom_A_epsilon_cur}).
Note that although in general for a given function $u\in L^2(Q,d\mu)$ there exists more than one pair $(u,\curl A^{1/2}u)\in H^1_{\curl A^{1/2}}(\R^3,d\mu^\varepsilon),$ it is easy to see that for each $u \in \rm{dom}(\mathcal{A}^\varepsilon)$ there exists a unique $\curl A^{1/2}u$ with the property described in \eqref{dom_A_epsilon_cur}.
Note that for each $\varepsilon>0$ the domain $\rm{dom}(\mathcal{A}^\varepsilon)$ is dense in 
$L^2(\R^3,d\mu^\varepsilon)\cap\{u: \divv\bigl(A^{-1/2}(\cdot/\varepsilon)u\bigr)=0\}.$
 Indeed, by the definition of $\rm{dom} (\mathcal{A}^\varepsilon)$, if $g\in L^2(\R^3,d\mu^\varepsilon)$, $\divv g=0$, and $u,v \in \rm{dom} (\mathcal{A}^\varepsilon)$ are such that $\mathcal{A}^\varepsilon u+ u= -A^{1/2}g$ and $\mathcal{A}^\varepsilon v+ v= -u$, then one has
\[
\int_{\R^3} |u|^2 d\mu^\varepsilon=\int_{\R^3}\sqrt{A(\cdot/\varepsilon)} g\cdot{v}d\mu^\varepsilon.
\]
It follows that for a vector $\sqrt{A(\cdot/\varepsilon)}g$ orthogonal to $\rm{dom} (\mathcal{A}^\varepsilon)$ one has $u=0$, and therefore $g=0,$ as required. Furthermore,
 the operator $\mathcal{A}^\varepsilon$ is symmetric with vanishing defect indices, hence it is self-adjoint.

Throughout the paper, all integrals and differential operators, unless indicated otherwise, are understood appropriately with respect to the measure $\mu.$ We use the notation $e_\varkappa$ for the exponent $\exp({\rm i}\varkappa\cdot y),$ $y\in Q,$ $\varkappa\in[-\pi,\pi)^3,$ and a similar notation $e_\theta$ for the exponent $\exp({\rm i}\theta\cdot x),$ $x\in{\mathbb R}^3,$ $\theta\in\varepsilon^{-1}[-\pi,\pi)^3.$ Furthermore, we denote by $C^1_{\#}$ denotes the set of continuously differentiable $Q$-periodic functions and by $C^1_{\#,0}$ its subset of functions having zero mean over $Q.$ 
Finally, the notation
$H^1_{\#}$ is used for the the set of $Q$-periodic functions in $H^1_{\rm loc}(\R^3,d\mu)$ and $H^1_{\#, 0}$ for its subset of functions with zero mean over $Q.$




\section{Sobolev spaces of quasiperiodic functions}

The aim of this section is to describe the functional analytic framework for our study of the problem (\ref{Max_general}).
As a particular case of the notion of ``weak differentiability" of square-integrable vector functions with respect an arbitrary Borel measure, 
we introduce a suitable generalisation of the 
classical curl operator. In what follows, $\mu$ is an arbitrary $Q$-periodic Borel measure. Denote by $L^2(Q, d\mu)$ the space of functions with values in ${\mathbb C}^3$ that are square integrable over $Q$ with respect to the measure 
$\mu.$

\begin{defi}
	\label{defiH1curl}
The space $H^1_{\curl A^{1/2}}$
is defined as the closure of the set 
\begin{equation}
\bigl\{\bigl(\phi,\curl(A^{1/2}\phi\bigr),\ \phi\in[C^1_\#]^3 
\bigr\}
\label{Hcurldef}
\end{equation}
 in the product $L^2(Q, d\mu)\times L^2(Q, d\mu).$
\end{defi}

Elements of the closure of (\ref{Hcurldef}) are pairs $(u,v)$ of $\mu$-measurable functions on $Q$ such that
\begin{equation}\label{propH1curl}
\exists\,\{\phi_n\}_{n=1}^\infty\subset\bigl[C^1_\#\bigr]^3: \quad \quad \int_Q |\phi_n-u|^2 d\mu\stackrel{n\to\infty}{\longrightarrow}0 \quad \quad \int_Q\bigl|\curl(A^{1/2}\phi_n)-v\bigr|^2 d\mu\stackrel{n\to\infty}{\longrightarrow}0.
\end{equation}
The element\footnote{For a general measure $\mu,$ a vector $u\in L^2(Q, d\mu)$ has multiple curls with respect to $\mu.$ In particular, any vector $g\in L^2(Q,d\mu)$ with the property
\begin{equation*}
\exists\,\{\phi_n\}\subset[C^1_\#]^3: \quad
\int_Q |\phi_n|^2d\mu\stackrel{n\to\infty}{\longrightarrow}0, \quad \quad \int_Q\bigl|g-\curl(A^{1/2}\phi_n)\bigr|^2d\mu\stackrel{n\to\infty}{\longrightarrow}0
\end{equation*}
is a curl with respect to $\mu$ of the zero vector, {\it i.e.} a ``curl of zero".} $v$ in \eqref{propH1curl} is referred to as an $A^{1/2}$-curl of $u$ with respect to $\mu.$ 
We will often use the notation $\curl(A^{1/2}u)$ without indicating the measure $\mu$ explicitly, assuming that it is clear from the context what the measure is.



We now extend to the vector setting (see {\it e.g.} \cite{CD18} for the scalar case) the definition of the Sobolev  space of quasiperiodic functions with respect to 
the measure $\mu.$

\begin{defi}
	\label{HcurlAdef}
For each $\varkappa \in [-\pi, \pi)^3=:Q'$, the space $H^1_{\curl A^{1/2},\varkappa}(Q, d\mu)$ is defined as the closure of the set 
$\bigl\{ \bigl(e_\varkappa \phi, \curl (e_\varkappa A^{1/2}\phi)\bigr): \phi\in [C^1_\#]^3\bigr\}$ 
with respect to the norm of $L^2(Q, d\mu)\times L^2(Q, d\mu).$ For $(u,v) \in H^1_{\curl A^{1/2}, \varkappa},$ we denote by $\curl(e_\varkappa A^{1/2}u)$ the second element $v$ in the pair, which we sometimes refer to as a ``$\varkappa$-curl of $u$." We will continue using the notation 
$H^1_{\curl A^{1/2}}$ (see Definition \ref{defiH1curl}) for the space $H^1_{\curl A^{1/2}, \varkappa}$ with $\varkappa=0.$
\end{defi}

Note that there may be different elements in $H^1_{\curl A^{1/2},\varkappa}(Q, d\mu)$ with the same first component.  Indeed, for any pair $(u,v) \in H^1_{\curl A^{1/2},\varkappa}$ and a vector function $w$ obtained as the limit in 
$L^2(Q, d\mu)$ of $\curl(e_\varkappa A^{1/2}\phi_n)$ for a sequence $\{\phi_n\}\subset [C_\#^1]^3$ converging to zero in $L^2(Q, d\mu)$, the element $(u, v+w)$ is also in $H^1_{\curl A^{1/2},\varkappa}.$ Furthermore, there is a natural one-to-one map between $H^1_{\curl A^{1/2},\varkappa}(Q, d\mu)$ and $H^1_{\curl A^{1/2}}(Q, d\mu)$. In fact, for any pair $(u,v)\in H^1_{\curl A^{1/2},\varkappa}(Q, d\mu)$ one has $(\overline{e}_\varkappa u, \overline{e}_\varkappa(v-{\rm i}\varkappa \times A^{1/2}u))\in H^1_{\curl A^{1/2}}(Q, d\mu)$, since
\[
\curl (A^{1/2}\phi_n)=\overline{e}_\varkappa\curl(e_\varkappa A^{1/2}\phi_n)-{\rm i}\varkappa\times A^{1/2}\phi_n
\]
for all sequences $\{\phi_n\}\subset [C^1_\#]^3$ such that $e_\varkappa \phi_n \to \overline{e}_\varkappa u$ and $\curl (e_\varkappa A^{1/2}\phi_n)\to\curl (A^{1/2}u)$  in $L^2(Q,d\mu)$ as $n\to\infty.$ Conversely, for every $(\widetilde{u}, \widetilde{v})\in H^1_{\curl A^{1/2}}(Q, d\mu)$ one has $\widetilde{v}=\overline{e}_\varkappa(v-{\rm i}\varkappa\times A^{1/2}u)$ for some $(u,v)\in H^1_{\curl A^{1/2},\varkappa}(Q, d\mu)$.

For every $\varkappa \in Q',$  we analyse the operator $\mathcal{A}_\varkappa$ with domain
\begin{equation}
\begin{aligned}
{\rm dom} (\mathcal{A}_\varkappa)= &\Bigl\{ u \in L^2(Q,d\mu): \; \exists \curl ( e_\varkappa A^{1/2} u) \; {\rm such\;  that} \;\\[0.4em]
& \int_Q\widetilde{A}\curl(e_\varkappa A^{1/2} u)\cdot\overline{\curl(e_\varkappa A^{1/2}\varphi)} d\mu+\int_Q u\cdot\overline{\varphi}d\mu=-\int_Q A^{1/2} G\cdot\overline{\varphi} \qquad \forall \varphi \in [C^1_\#]^3\\[0.4em]
&\quad\qquad\qquad\qquad{\rm for \; some} \; G\in L^2(Q, d\mu), \; \overline{e}_\varkappa\divv (e_\varkappa G)=0\Bigr\},
\end{aligned}
\label{uG_link}
\end{equation}
and defined by 
$\mathcal{A}_\varkappa u= -A^{1/2}G-u,$
where $G\in L^2(Q,d\mu)$ and $u\in {\rm dom} (\mathcal{A}_\varkappa)$ are linked by (\ref{uG_link}). By an argument similar to that for $\mathcal{A}^\varepsilon,$ we infer that ${\rm dom}(\mathcal{A}_\varkappa)$ is dense in 
$
L^2(Q,d\mu)\cap\{u: \overline{e}_\varkappa\divv(e_\varkappa A^{-1/2} u)=0\}
$
and $\mathcal{A}_\varkappa$ is self-adjoint. 

\section{Floquet transform}
\label{sec2}
Here we recall the definition of the Floquet transform, see \cite{CDO22}. For $\varepsilon>0$, the $\varepsilon$-Floquet transform $\mathcal{F}_\varepsilon $ is defined on 
functions $u\in L^2(\R^3,d\mu^\varepsilon),$ with compact support by the formula
\[
(\mathcal{F}_\varepsilon u)(y, \theta)= \left(\frac{\varepsilon^2}{2\pi}\right)^{3/2}\sum_{n\in \Z^3} u(\varepsilon y+\varepsilon n) \exp (-{\rm i}\varepsilon n\cdot \theta), \qquad y\in Q, \quad \theta \in \varepsilon^{-1}Q'=\varepsilon^{-1}[-\pi,\pi)^3.
\]
The mapping $\mathcal{F}_\varepsilon$ preserves the norm and can be extended to a unitary transform from $L^2(\R^3,d\mu^\varepsilon)$ to $L^2(\varepsilon^{-1}Q'\times Q, d\theta\times d\mu),$ for which we keep the same notation ${\mathcal F}_\varepsilon$ and the term ``$\varepsilon$-Floquet transform". Its inverse 
is given by
\begin{equation}
(\mathcal{F}_\varepsilon^{-1}g)(x)=({2\pi})^{-3/2} \int_{\varepsilon^{-1}Q'} g\biggl(\frac{x}{\varepsilon}, \theta\biggr) d\theta,\quad x\in{\mathbb R}^3, \qquad g\in L^2(Q\times\varepsilon^{-1}Q', d\mu\times d\theta),
\label{inverseFloquet}
\end{equation}
where for each $\theta\in \varepsilon^{-1}Q'$ the function $g\in L^2(Q\times\varepsilon^{-1}Q', d\mu\times d\theta)$ is extended to the whole of ${\mathbb R}^3$ so that 
$g(z, \theta)=\widetilde{g}(z,\theta)\exp({\rm i}z\cdot\theta),$ 
$z\in{\mathbb R}^3,$
where $\widetilde{g}(\cdot, \theta)$ is $Q$-periodic.

As a result of applying the transform ${\mathcal F}_\varepsilon$ to the operator ${\mathcal A}_\varepsilon$ of the problem (\ref{Max_general_g}), we obtain the following representation for the resolvent of ${\mathcal A}_\varepsilon.$

\begin{prop}
For each $\varepsilon>0$ the following unitary equivalence between the resolvent of $\mathcal{A}^\varepsilon$ and the direct integral of the resolvents of $\mathcal{A}_{\varepsilon\theta},$  $\theta\in \varepsilon^{-1}Q',$ holds:
\begin{equation*}
(\mathcal{A}^\varepsilon+I)^{-1}= \mathcal{F}_\varepsilon^{-1}\biggl(\int^\oplus_{\varepsilon^{-1}Q'} e_\varkappa (\varepsilon^{-2} \mathcal{A}_{\varepsilon\theta} +I)^{-1} \overline{e}_{\varepsilon\theta} d\theta\biggr)\mathcal{F}_\varepsilon,
\end{equation*}
where $e_{\varepsilon\theta},$ $\overline{e}_{\varepsilon\theta}$ represent the operators of multiplication by $e_{\varepsilon\theta},$ $\overline{e}_{\varepsilon\theta},$ respectively.
\end{prop}

\begin{proof}[Sketch of proof]
The argument is similar to the one discussed in \cite{CDO22} for the Maxwell system in the absence of external currents. Let us consider the solution 
$(\mathfrak{D}^\varepsilon, \curl(\sqrt{A(\cdot/\varepsilon)} \mathfrak{D}^\varepsilon)) \in H^1_{\curl A^{1/2}}(Q,d\mu)$ of the problem \eqref{Max_general_g} with $J\in[C^\infty_0(\R^3)]^3.$ For such $\mathfrak{D}^\varepsilon$ we denote the ``periodic amplitude" of its Floquet transform as follows:
\begin{equation}
{\mathfrak D}^\varepsilon_\theta(y):=\overline{e}_{\varepsilon\theta}\mathcal{F}_\varepsilon\mathfrak{D}^\varepsilon= \left( \frac{\varepsilon^2}{2\pi}\right)^{3/2} \sum_{n\in \Z^3} \mathfrak{D}^\varepsilon(\varepsilon y+\varepsilon n)\exp\bigl(-{\rm i}(\varepsilon y+\varepsilon n)\cdot \theta\bigr), \quad y\in Q.
\label{Floquet_def1}
\end{equation}
For any choice of $\curl(A^{1/2} \mathfrak{D}^\varepsilon),$ in particular for the one in \eqref{Max_general_g}, the function
\[
\curl(e_{\varepsilon\theta}A^{1/2}{\mathfrak D}^\varepsilon_\theta)(y)=\varepsilon \left( \frac{\varepsilon^2}{2\pi}\right)^{3/2} \sum_{n\in \Z^3} \curl(A^{1/2} \mathfrak{D}^\varepsilon)(\varepsilon y+\varepsilon n)\exp\bigl(-{\rm i}(\varepsilon y+\varepsilon n)\cdot \theta\bigr), \quad y\in Q,
\]
is an ``$A^{1/2}$-curl'' of $e_{\varepsilon\theta}{\mathfrak D}^\varepsilon_\theta$ in sense that $(e_{\varepsilon\theta}{\mathfrak D}^\varepsilon_\theta, \curl (e_{\varepsilon\theta}A^{1/2}{\mathfrak D}^\varepsilon_\theta))\in H^1_{\curl  A^{1/2},\varepsilon\theta}$. 

Under the Floquet transform, the formulation (\ref{Max_general_g_weak}) is converted into
\begin{equation}
\begin{aligned}
\varepsilon^{-2}\int_Q\widetilde{A}\curl(e_\varkappa A^{1/2} {\mathfrak D}^\varepsilon_\theta)\cdot\overline{\curl(e_\varkappa A^{1/2} \varphi)} d\mu+\int_Q {\mathfrak D}^\varepsilon_\theta\cdot\overline{\varphi} d\mu =-\int_Q A^{1/2}{J}_\theta^\varepsilon\cdot\overline{\varphi} d\mu\qquad\ \ \forall \varphi\in\bigl[C_\#^1\bigr]^3
\end{aligned}
\label{Max_ft_weak}
\end{equation}
where ${J}_\theta^\varepsilon:=\overline{e}_{\varepsilon\theta}\mathcal{F}_\varepsilon J$ is such that $\overline{e}_{\varepsilon\theta}\divv (e_{\varepsilon\theta}{J}_\theta^\varepsilon)=0,$ in the sense of the identity
\begin{equation}
\label{div_ekG}
\int_Q e_{\varepsilon\theta}{J}_\theta^\varepsilon\cdot\overline{\nabla(e_{\varepsilon\theta}\varphi)}=0 \qquad \forall\varphi\in C^1_\#.
\end{equation}
The density of $C^1_0(\R^3)$ in $L^2(\R^3,d\mu)$ implies the claim.
\end{proof}

	For each $\varepsilon>0,$ $\theta\in\varepsilon^{-1}Q',$ we study the behaviour of the solution ${\mathfrak D}^\varepsilon_\theta$ to the problem (\ref{Max_ft_weak}), with ${J}_\theta^\varepsilon$ replaced by an arbitrary ${G}\in L^2(Q,d\mu)$ satisfying the condition ({\it cf.} (\ref{div_ekG})) 
$\overline{e}_{\varepsilon\theta}\divv (e_{\varepsilon\theta}{G})=0.$
For brevity, we also write the mentioned problem in the ``strong" form 
\begin{equation}
\label{Max_ft}
\varepsilon^{-2} A^{1/2} \overline{e}_{\varepsilon\theta}\curl\widetilde{A} \curl\bigl(e_{\varepsilon\theta}A^{1/2} {\mathfrak D}^\varepsilon_\theta\bigr)+{\mathfrak D}^\varepsilon_\theta= -A^{1/2}{G}.
\end{equation}

\section{An analogue of Helmholtz decomposition}
\label{Helm_sec}

The Helmholtz decomposition of the space of square-integrable functions is an important tool for the analysis of the Maxwell system, see {\it e.g.} \cite{JKO}. In the present work we develop its version for spaces of quasiperiodic functions with respect to an arbitrary periodic Borel measure 
$\mu,$ taking into account the structure of the problem \eqref{Max_ft}. 

Before formulating the next proposition, we recall that the notions  of a gradient of a quasiperiodic function that is square integrable with respect the measure $\mu$ and of the associated Sobolev spaces $H^1_\varkappa(Q,d\mu)$ can be defined, see \cite{CD18} for details. Furthermore, we say that $v\in L^2(Q, d\mu)$ is solenoidal, or $\overline{e}_\varkappa\divv (e_\varkappa {A}^{-1/2})$-free, if 
\begin{equation}
\int_Q {A}^{-1/2} e_\varkappa v\cdot\overline{\nabla (e_\varkappa \varphi)} d\mu=0\qquad \forall \varphi \in C^1_{\#,0}.
\label{def_solenoidal}
\end{equation}

We begin by introducing a problem for a scalar ``potential" $\Phi_w^{(\varkappa)},$ whose appropriate gradient plays the role of the irrotational part of a function $w$ in the classical Helmholtz decomposition.
\begin{prop}
\label{Phi_prop}
For each $\varkappa\in{\mathbb R}^3,$ $w\in L^2(Q, d\mu),$ there is a unique solution $\Phi_w^{(\varkappa)}\in H^1_{\#,0}$ to the problem
\begin{equation}
\overline{e}_\varkappa\divv\bigl({A}^{-1} \nabla\bigl(e_\varkappa \Phi_w^{(\varkappa)}\bigr)\bigr)= \overline{e}_\varkappa\divv(e_\varkappa {A}^{-1/2} w),
\label{Phi_kappa}
\end{equation}
understood in the sense of the integral identity
\begin{equation}\label{phi_eq_kneq0weak}
\int_Q {A}^{-1} \nabla\bigl(e_\varkappa \Phi_w^{(\varkappa)}\bigr) \cdot\overline{\nabla (e_\varkappa \varphi)} =\int_Q {A}^{-1/2} e_\varkappa w\cdot\overline{\nabla (e_\varkappa \varphi)} \qquad \forall \varphi \in C^1_{\#,0}.
\end{equation}
\end{prop}
\begin{proof}
The left-hand side of \eqref{phi_eq_kneq0weak} defines a sesquilinear form that is bounded and coercive on $H^1_{\#,0}.$ The coercivity follows from the Poincar\'{e}-type inequality discussed in \cite[Section 5]{CD18} for the scalar case. Bearing in mind that the right-hand side of \eqref{phi_eq_kneq0weak} is a bounded linear functional on $H^1_{\#,0},$ we apply the Riesz representation theorem to infer the existence and uniqueness of a solution.
\end{proof}


	 For each $\varkappa\in{\mathbb R}^3,$ consider the vector $\Psi_\varkappa$ with entries in $H^1_{\#,0}$ such that\footnote{In the case when $\mu$ is the Lebesgue measure, equation (\ref{eq_Psikappa}) takes the form 
	 $(\nabla+{\rm i}\varkappa)\cdot A^{-1}((\nabla+{\rm i}\varkappa)\Psi_\varkappa+I)=0.$ Recall that the widely known unit-cell problem in homogenisation (for the matrix $A^{-1}$) is $\nabla\cdot A^{-1}(\nabla\Psi+I)=0$ --- see, e.g., \cite[Section 1.2]{JKO}.}
\begin{equation}
\label{eq_Psikappa}
\overline{e}_\varkappa\divv\bigl\{ A^{-1}\bigl(\nabla(e_\varkappa\Psi_\varkappa) + e_\varkappa I\bigr)\bigr\}=0,
\end{equation}
where $(\nabla(e_\varkappa\Psi_\varkappa))_{ij}:= (e_\varkappa\Psi_\varkappa)_{j,i},$ $i,j=1,2,3,$ so that the $j$-th component of $\Psi_\varkappa$ coincides with $\Phi_{A^{-1/2}e_j}^{(\varkappa)},$ denoting by $e_j$ the vectors with components $(e_j)_{i}=\delta_{ij}.$ 
For each $w\in L^2(Q, d\mu),$ 
we write
\begin{equation}
\label{decokneq0}
w=\widetilde{w}_{\varkappa}+{A}^{-1/2}\bigl(\overline{e}_\varkappa\nabla(e_\varkappa\Psi_\varkappa)+I\bigr)\gamma+{A}^{-1/2}\overline{e}_\varkappa\nabla\bigl(e_\varkappa \Phi_w^{(\varkappa)}\bigr),
\end{equation}
where $\widetilde{w}_{\varkappa}$ is solenoidal, and the constant $\gamma$ (which depends on $w$ and $\varkappa$ and is sometimes denoted by $\gamma_w$ in what follows) is chosen so that
\begin{equation}
\int_Q A^{-1/2}\widetilde{w}_{\varkappa}=0.
\label{mean_wtilde}
\end{equation}
The representation (\ref{decokneq0}) is unique. Indeed, Proposition \ref{Phi_prop} provides a unique function $\Phi_w^{(\varkappa)}$ with zero mean such that $w-{A}^{-1/2}\overline{e}_\varkappa\nabla(e_\varkappa \Phi_w^{(\varkappa)})$ is solenoidal, and the constant $\gamma$ is determined uniquely from the requirement (\ref{mean_wtilde}) by the formula
\begin{equation}
\gamma=\bigl(A^{\rm hom}_\varkappa\bigr)^{-1}
\int_Q\Bigl\{A^{-1/2}w-A^{-1}\overline{e}_\varkappa\nabla\bigl(e_\varkappa \Phi_w^{(\varkappa)}\bigr)\Bigr\},
\label{c_form}
\end{equation}
where 
\begin{equation}
	\label{A_hom_hat}
		A^{\rm hom}_\varkappa
	:=\int_Q A^{-1}\bigl(\overline{e}_\varkappa\nabla(e_\varkappa\Psi_\varkappa)+I\bigr),\qquad \varkappa\in Q'.
\end{equation}

\begin{prop}
	\label{lemma_boundAhom}
	There exist constants $\nu_1,$ $\nu_2>0$ independent of $\varkappa\in Q'$, such that 
	\begin{equation}
		\label{bound_A_hom}
		\nu_1I\leq A^{\rm hom}_{\varkappa}\leq \nu_2I.
	\end{equation}
\end{prop}
\begin{proof}
	 The argument is similar to that for the classical Voigt-Reiss inequality (see \cite[Section 1.6]{JKO}). The quadratic form of the matrix $A^{\rm hom}_{\varkappa}$ defined by (\ref{A_hom_hat}) is given by
	\begin{equation}
		A^{\rm hom}_{\varkappa} \lambda\cdot \lambda=\inf_{\psi\in H^1_{\#,0}} \int_Q A^{-1}\bigl(\overline{e}_\varkappa\nabla(e_\varkappa\psi)+\lambda\bigr)\cdot{\bigl(e_\varkappa\overline{\nabla(e_\varkappa\psi)}+\lambda\bigr)}, \quad \lambda\in \R^3,
		\label{lambda_form}
	\end{equation}
due to the fact that the minimum in (\ref{lambda_form}) is attained on $\psi=\Psi_\varkappa\cdot\lambda.$ 
	Setting $\psi=0$ in the integral on the right-hand side of (\ref{lambda_form}), the upper bound in (\ref{bound_A_hom}) follows 
	with $\nu_2:= \|A^{-1}\|_{L^\infty(Q)}.$  
	
	To prove the lower bound, we use a variational characterisation for $\big(A^{\rm hom}_{\varkappa}\big)^{-1},$ namely 
	\begin{equation}
		\label{A_homkappa_inverse}
		\big(A^{\rm hom}_{\varkappa}\big)^{-1} \xi \cdot \xi=
		\inf_{{\mathcal X}_{\varkappa, \xi}}\Vert v\Vert^2_{L^2(Q,d\mu)}, \qquad  \xi \in \R^3,
	\end{equation}
	where 
	\[
	{\mathcal X}_{\varkappa,\xi}:=\biggl\{v\in L^2(Q, d\mu): \overline{e}_\varkappa{\rm div}\bigl(e_\varkappa A^{-1/2}v\bigr)=0, \int_QA^{-1/2}v=\xi\biggr\}.
	\]
	Indeed, for $\lambda=\bigl(A^{\rm hom}_{\varkappa}\bigr)^{-1}\xi$ the function  $v_{\varkappa,\lambda}:={A}^{-1/2}(\overline{e}_\varkappa\nabla(e_\varkappa\Psi_\varkappa)+I)\lambda$ is an element of ${\mathcal X}_{\varkappa,\xi}$ and
 satisfies
	\begin{equation*}
		\int_Qv_{\varkappa,\lambda}\cdot\overline{\varphi}=0 \qquad \forall\varphi\in{\mathcal X}_{\varkappa, 0}.
	\end{equation*}
Hence, it is the (unique) minimiser for (\ref{A_homkappa_inverse}). Combining this with \eqref{lambda_form} yields the claim.

	The required lower bound is now obtained by noting that the right-hand side of \eqref{A_homkappa_inverse} is a continuous function of $\varkappa\in Q',$ $\xi\in\{\xi\in{\mathbb R}^3: |\xi|=1\}$ and hence attains a maximum value. We then set $\nu_1$ in \eqref{bound_A_hom} to be the inverse of this maximum value.
\end{proof}

\begin{rmk}
Both $\widetilde{w}_{\varkappa}$ and the second term in (\ref{decokneq0}) are solenoidal, and the latter accounts for the mean value of $A^{1/2}w$ so (\ref{mean_wtilde}) holds. Notice also that the three terms in (\ref{decokneq0}) are pairwise orthogonal in 
$L^2(Q, d\mu)$.
\end{rmk}


\section{Poincar\'{e}-type inequality}

Here we prove a version of the Poincar\'{e} inequality for functions in 
 $H^1_{\curl A^{1/2},\varkappa}(Q,d\mu),$ see Definition \ref{HcurlAdef}. 
Before formulating it, we make a simple observation concerning a subset of $H^1_{\curl A^{1/2},\varkappa}(Q,d\mu).$ 
\begin{lem}
\label{zero_lemma}
For any $\eta\in H^1_{\#,0}(Q, d\mu),$ the zero vector is one of the curls of $A^{-1/2}\overline{e}_\varkappa\nabla(e_\varkappa \eta),$ {\it i.e.} one has 
$\bigl(A^{-1/2}\overline{e}_\varkappa\nabla (e_\varkappa \eta), 0\bigr)\in H^1_{\curl A^{1/2},\varkappa}(Q, d\mu).$
\end{lem}
\begin{proof}
	The statement follows immediately from Definition \ref{HcurlAdef}.
\end{proof}
\begin{thm}
\label{thm_curlpoincare}

Suppose that there exists $C_{\rm P}>0$  such that\,\footnote{In \cite[Appendix B]{CDO22} we describe a class of measure that satisfy this assumption.}
\begin{equation}
\begin{aligned}
\biggl\|u-\int_Qu\biggr\|_{L^2(Q,d\mu)}&\leq C_{\rm P}\bigl\|\curl(e_\varkappa u)\bigr\|_{L^2(Q,d\mu)}\\[0.3em]
&\forall \varkappa\in Q',\quad\bigl(e_\varkappa u, \curl(e_\varkappa u)\bigr)\in H^1_{\curl,\varkappa}(Q, d\mu)\ \ \ {\rm such\ that}
\ \ \ \overline{e}_\varkappa{\rm div}(e_\varkappa u)=0,
\end{aligned}
\label{curlpoincare}
\end{equation}
and assume in addition that 
\begin{equation}
\max_Q{\bigl\Vert A^{-1}\nabla A\bigr\Vert}<1.
\label{ass_a}
\end{equation}
For each $w\in L^2(Q, d\mu),$ $\varkappa\in Q',$ define $\widetilde{w}_{\varkappa}$ as in (\ref{decokneq0}). 
There exists $C>0$ such that 
\begin{equation}
\begin{aligned}
\bigl\Vert\widetilde{w}_{\varkappa}\bigr\Vert_{L^2(Q, d\mu)}&\leq C\bigl\|\curl (e_\varkappa A^{1/2} w)\bigr\|_{L^2(Q, d\mu)}
\quad \forall \varkappa \in Q', \ \bigl(e_\varkappa w, \curl (e_\varkappa A^{1/2} w)\bigr)\in H^1_{\curl A^{1/2},\varkappa}(Q, d\mu).
\end{aligned} 
\label{poincare_ineq_A}
\end{equation}
\end{thm}
\begin{proof}
The proof is based on applying (\ref{curlpoincare}) to a pair 
$\bigl(e_\varkappa A^{-1/2}\widetilde{w}_{\varkappa}, \curl (e_\varkappa A^{-1/2}\widetilde{w}_{\varkappa})\bigr)$
generated by the function $w$ in the inequality \eqref{poincare_ineq_A}.

In order to define $\curl (e_\varkappa A^{-1/2}\widetilde{w}_{\varkappa})\bigr)$
as an element of $L^2(Q,d\mu),$
take a sequence $\{\zeta_n\}\subset[C^1_\#]^3$ such that 
$
\zeta_n\rightarrow w
$
and 
$\curl (e_\varkappa A^{1/2}\zeta_n)\rightarrow \curl(e_\varkappa A^{1/2}w)$ in $L^2(Q,d\mu).$
(Such a sequence exists by the definition of $H^1_{\curl A^{1/2},\varkappa}$ --- see Definition \ref{HcurlAdef}.)
Next, consider the sequence 
\[
\phi_n= \zeta_n - A^{-1/2}\bigl(\overline{e}_\varkappa \nabla (e_\varkappa\Psi_\varkappa)+I\bigr)\gamma_n - A^{-1/2} \overline{e}_\varkappa \nabla\bigl(e_\varkappa\Phi_{\zeta_n}^{(\varkappa)}\bigr),\qquad n\in{\mathbb N},
\]
where $\gamma_n$ is given by (\ref{c_form}) with $w$ replaced by $\zeta_n.$
By the decomposition \eqref{decokneq0}, one has the convergence $\phi_n \to \widetilde{w}_{\varkappa}$ in $L^2(Q,d\mu).$ 
For each $n\in{\mathbb N},$ we write
\begin{equation}
	\begin{aligned}
\curl\bigl(e_\varkappa A^{-1/2}\phi_n\bigr)&=\curl\bigl(e_\varkappa A^{-1}A^{1/2}\phi_n\bigr)=A^{-1}\curl\bigl(e_\varkappa A^{1/2}\phi_n\bigr)+(\nabla A^{-1}) A\times A^{-1/2} e_\varkappa \phi_n\\[0.3em]
&=A^{-1}\curl\bigl(e_\varkappa A^{1/2}\zeta_n\bigr)-{\rm i}A^{-1}e_\varkappa\varkappa\times\gamma_n-A^{-1}\nabla A\times A^{-1/2} e_\varkappa \phi_n.
\end{aligned}
\label{lim_n_curl}
\end{equation}
Passing to the limit in (\ref{lim_n_curl}) as $n\to\infty$, we define
\begin{equation}
	\label{definition_curlwtilde}
\curl\bigl(e_\varkappa A^{-1/2}\widetilde{w}_{\varkappa}\bigr):=A^{-1}\curl \bigl(e_\varkappa A^{1/2}w\bigr)-{\rm i}A^{-1}e_\varkappa\varkappa\times\gamma +A^{-1}\nabla A\times e_\varkappa A^{-1/2} \widetilde{w}_{\varkappa},
\end{equation}
where $\gamma$ is given by \eqref{c_form}.

By virtue of (\ref{curlpoincare}), 
we have 
\begin{align}
\bigl\| A^{-1/2}\widetilde{w}_{\varkappa}
\bigr\|_{L^2(Q,d\mu)}
\leq C_{\rm P}\bigl\|\curl\bigl(e_\varkappa A^{-1/2}\widetilde{w}_{\varkappa}\bigr)\bigr\|_{L^2(Q,d\mu)}.
\label{pre_est}
\end{align}
\begin{lem}
	\label{second_lemma}
For all $\varkappa\in Q',$ one has 
\[
\Vert A^{-1/2}e_\varkappa\varkappa\times\gamma\bigr\Vert_{L^2(Q,d\mu)}\le\bigl\Vert\curl(e_\varkappa A^{-1/2}\gamma)\Vert_{L^2(Q,d\mu)}\le\bigl\Vert\curl(e_\varkappa A^{-1/2}w)\Vert_{L^2(Q,d\mu)}.
\]
\end{lem}
\begin{proof}
	Note that 
	\[
	\curl(e_\varkappa A^{-1/2}\gamma)=A^{-1/2}e_\varkappa\gamma\times\biggl(\frac{1}{2}A^{-1}\nabla A-{\rm i}\varkappa\biggr).
	\]
	Hence, 
	\[
	\bigl\Vert\curl(e_\varkappa A^{-1/2}\gamma)\bigr\Vert^2_{L^2(Q,d\mu)}=\biggl\Vert A^{-1/2}e_\varkappa\gamma\times\frac{1}{2}A^{-1}\nabla A\biggr\Vert^2_{L^2(Q,d\mu)}+
	\bigl\Vert A^{-1/2}e_\varkappa\gamma\times(-{\rm i}\varkappa)\bigr\Vert^2_{L^2(Q, d\mu)},
	\]
	from which the claim follows. \end{proof}
Substituting the expression (\ref{definition_curlwtilde}) into the right-hand side of (\ref{pre_est}), taking into account Lemma \ref{second_lemma},
and using the assumption (\ref{ass_a}) yields
\begin{align*}
	\bigl\| A^{-1/2}\widetilde{w}_{\varkappa}
	\bigr\|_{L^2(Q,d\mu)}
	\leq\widetilde{C}\bigl\|\curl\bigl(e_\varkappa A^{1/2}w\bigr)\bigr\|_{L^2(Q,d\mu)},\ \ \qquad \widetilde{C}:=\frac{C_{\rm P}\max_Q\Vert A^{-1}\Vert}{1-\max_Q{\bigl\Vert(\nabla A^{-1})A\bigr\Vert}},
\end{align*}
and hence \eqref{poincare_ineq_A} holds with $C=\widetilde{C}\Vert A^{-1/2}\Vert^{-1}.$
\end{proof}

\section{Electric displacement $D^\varepsilon$}
\label{displacement_sec}

In order to write the asymptotic approximation for the electric displacement ${\mathfrak D}^\varepsilon_\theta,$ we consider the following generalised cell problem for
 a ${\rm div}A^{-1/2}$-free (see (\ref{def_solenoidal}))
matrix function $\widetilde{N}$ with columns in $H^1_{\curl A^{1/2}}(Q,d\mu):$ 
\begin{equation}
A^{1/2} \curl \widetilde{A} \curl\bigl(A^{1/2} \widetilde{N}\bigr)=-A^{1/2} \curl \widetilde{A},
\qquad
\int_Q A^{-1/2} \widetilde{N}=0,
\label{cell_pbN_general}
\end{equation}
understood in the sense of the integral identity
\begin{equation}
	\label{cell_pbN_general_weak}
\int_Q \widetilde{A} \curl (A^{1/2}\widetilde{N})\cdot\overline{\curl (A^{1/2}\varphi)}d\mu=-\int_Q \widetilde{A}\,\overline{\curl (A^{1/2}\varphi)} \qquad \forall \varphi \in H^1_{\curl A^{1/2}}(Q, d\mu),\ \  \int_Q A^{-1/2}\varphi=0.
\end{equation}

%
%
%

The Poincar\'{e} inequality \eqref{poincare_ineq_A} immediately implies the following statement.
\begin{prop}
There exists a unique ${\rm div}A^{-1/2}$-free solution $\widetilde{N}$
to  
\eqref{cell_pbN_general}.
\end{prop}

 For each $\varepsilon>0,$ $\theta\in\varepsilon^{-1}Q',$ $G\in L^2(Q,d\mu),$ define $d^\varepsilon_\theta(G)\in \C^3$ by the formula
	\begin{equation}\label{dtheta_eq_gen}
		d^\varepsilon_\theta(G)= 
		-\big(\widetilde{\mathfrak A}^{\rm hom}_\theta
		+A^{\rm hom}_{\varepsilon\theta} \big)^{-1} 
		\int_Q G.
	\end{equation}	
Here, for each $\theta \in \R^3$ the matrix 
$\widetilde{\mathfrak A}^{\rm hom}_\theta$ has entries 
\begin{equation}
	(\widetilde{\mathfrak A}^{\rm hom}_\theta)_{ij}=\varepsilon_{isk}({\rm i}\theta)_s\widetilde{A}^{\rm hom}_{kr}\varepsilon_{rtj}({\rm i}\theta)_t,
	\qquad i,j=1,2,3,
	\label{Atheta_tilde}
\end{equation}
where
\begin{equation}
\widetilde{A}^{\rm hom}:= \int_Q \widetilde{A}\bigl(\curl(A^{1/2}\widetilde{N})+I\bigr).
\label{Ahom_tilde_def}
\end{equation} 
\begin{thm}
\label{main_thm_gen}
There exists $C>0$ independent of $\varepsilon$, $\theta$ and $G$, such that for the solution ${\mathfrak D}^\varepsilon_\theta$ of \eqref{Max_ft} one has 
\begin{equation}\label{main_estimate_gen}
\bigl\| {\mathfrak D}^\varepsilon_\theta -A^{-1/2}(\overline{e}_{\varepsilon\theta}\nabla (e_{\varepsilon\theta} \Psi_{\varepsilon\theta})+I){d}^\varepsilon_\theta(G)\bigr\|_{L^2(Q, d\mu)}\leq C \varepsilon \| G\|_{L^2(Q, d\mu)}.
\end{equation}
\end{thm}
The proof of Theorem \ref{main_thm_gen} is provided in Section \ref{est_sec} below.

To obtain a norm-resolvent estimate in the whole-space setting for the original problem \eqref{Max_general_g}, it remains to apply the inverse Floquet transform to the estimate \eqref{main_estimate_gen}. Note that for  $J\in L^2(\R^3,d\mu^\varepsilon),$ 
one has
\begin{equation}
\int_Q J^\varepsilon_\theta d\mu= \widehat{J}(\theta), \quad \varepsilon>0,\quad\theta \in \varepsilon^{-1} Q',
\label{Jmean}
\end{equation} 
where $J_\theta^\varepsilon$ is defined analogously to (\ref{Floquet_def1}), and
\begin{equation*}
	\widehat{J}(\theta):= (2\pi)^{-3/2}\int_{\R^3}J\overline{e_\theta} \; d\mu^\varepsilon, \quad \theta \in \R^3,
\end{equation*}
is the Fourier transform of $J.$ Note that (\ref{div_ekG}) implies that $J^\varepsilon_\theta\cdot\theta=0$ for all $\varepsilon>0,$ $\theta\in\varepsilon^{-1}Q',$ and hence 
$\widehat{J}(\theta)\cdot\theta=0$ for all $\theta\in{\mathbb R}^3.$

\begin{cor}
There exists a constant $C>0$ such that for all $\varepsilon>0,$ $J\in L^2(\R^3,d\mu^\varepsilon),$ one has
\begin{equation*}
\begin{aligned}
\biggl\|{D}^\varepsilon +(2\pi)^{-3/2}A^{-1}\biggl(\frac{\cdot}{\varepsilon}\biggr)\int_{\R^3} 
\biggl\{\overline{e}_{\varepsilon\theta}\nabla({e}_{\varepsilon\theta}\Psi_{\varepsilon\theta})\biggl(\frac{\cdot}{\varepsilon}\biggr)+I\biggr\}
&\big(\mathfrak{\widetilde{A}}_\theta^{\rm hom}+A^{\rm hom}_{\varepsilon\theta}\big)^{-1} \widehat{J}(\theta) e_\theta  d\theta \biggr\|_{L^2(\R^3, d\mu^\varepsilon)}\\[0.5em]
&\leq C \varepsilon \|J\|_{L^2(\R^3, d\mu^\varepsilon)}.
\end{aligned}
\end{equation*}
\end{cor}

\begin{proof}
	Note that the solution  ${\mathfrak D}^\varepsilon_\theta$ to \eqref{Max_ft} with $G=J^\varepsilon_\theta$ on its right-hand side is the Floquet transform (see Section \ref{sec2}) of ${\mathfrak D}^\varepsilon=\sqrt{A(\cdot/\varepsilon)}D^\varepsilon.$ Therefore, one has 
	\begin{equation}
		\begin{aligned}
			\mathfrak{D}^\varepsilon&+(2\pi)^{-3/2}A^{-1/2}\biggl(\frac{\cdot}{\varepsilon}\biggr)\int_{\R^3} 
			\biggl\{\overline{e}_{\varepsilon\theta}\nabla({e}_{\varepsilon\theta}\Psi_{\varepsilon\theta})\biggl(\frac{\cdot}{\varepsilon}\biggr)+I\biggr\}
			{d}_\theta^\varepsilon(J^\varepsilon_\theta)d\theta\\[0.5em]
			&=\Big[ \mathcal{F}_\varepsilon^{-1}\bigl\{e_{\varepsilon\theta}{\mathfrak D}^\varepsilon_\theta\bigr\}- 
			\mathcal{F}_\varepsilon^{-1}\bigl\{e_{\varepsilon\theta} A^{-1/2}\bigl(\overline{e}_{\varepsilon\theta}\nabla({e}_{\varepsilon\theta}\Psi_{\varepsilon\theta})+I\bigr) {d}_\theta^\varepsilon(J^\varepsilon_\theta)\bigr\}\Big]\\[0.4em]
			&\hspace{2.5em}+\bigg[ \mathcal{F}_\varepsilon^{-1}\bigl\{e_{\varepsilon\theta} A^{-1/2}\bigl(\overline{e}_{\varepsilon\theta}\nabla({e}_{\varepsilon\theta}\Psi_{\varepsilon\theta})+I\bigr) {d}_\theta^\varepsilon(J^\varepsilon_\theta)\bigr\} \\[0.2em]
			&\hspace{5em}+(2\pi)^{-3/2}A^{-1/2}\biggl(\frac{\cdot}{\varepsilon}\biggr)\int_{\R^3} 
			\biggl\{\overline{e}_{\varepsilon\theta}\nabla({e}_{\varepsilon\theta}\Psi_{\varepsilon\theta})\biggl(\frac{\cdot}{\varepsilon}\biggr)+I\biggr\}
			\big(\widetilde{\mathfrak{A}}^{\rm hom}_\theta+A^{\rm hom}_{\varepsilon\theta}\big)^{-1} \widehat{J}(\theta) e_\theta(\cdot)d\theta\biggr].
		\end{aligned}
		\label{split_L2}
	\end{equation} 
	We analyse the $L^2$ norm of the expression (\ref{split_L2}). In view of the result of Theorem \ref{main_thm_gen} and the fact that $\mathcal{F}_\varepsilon$ is unitary, 
	we estimate the term in the first square brackets on the right-hand side of (\ref{split_L2}) as follows:
	\begin{align}\label{estimate_firstbit}
		\Bigl\|\mathcal{F}_\varepsilon^{-1}\Bigl\{e_{\varepsilon\theta}\bigl({\mathfrak D}^\varepsilon_\theta - 
		A^{-1/2}\bigl(\overline{e}_{\varepsilon\theta}\nabla({e}_{\varepsilon\theta}\Psi_{\varepsilon\theta})+I\bigr) {d}_\theta^\varepsilon(J^\varepsilon_\theta)\bigr)\Bigr\}\Bigr\|_{L^2(\R^3,d\mu^\varepsilon)}\leq C\varepsilon\bigl\|J\|_{L^2(\R^3,d\mu^\varepsilon)}.
	\end{align}
	Noting that (see (\ref{inverseFloquet}), (\ref{Jmean}))
	\begin{align*}
		&
		\mathcal{F}_\varepsilon^{-1}\bigl\{e_{\varepsilon\theta} A^{-1/2}\bigl(\overline{e}_{\varepsilon\theta}\nabla({e}_{\varepsilon\theta}\Psi_{\varepsilon\theta})+I\bigr) {d}_\theta^\varepsilon(J^\varepsilon_\theta)\bigr\}
		\\[0.4em]
		&
		\hspace{6.5em}=-(2\pi)^{-3/2}A^{-1/2}\biggl(\frac{\cdot}{\varepsilon}\biggr)\int_{\varepsilon^{-1}Q'} 
		\biggl\{\overline{e}_{\varepsilon\theta}\nabla({e}_{\varepsilon\theta}\Psi_{\varepsilon\theta})\biggl(\frac{\cdot}{\varepsilon}\biggr)+I\biggr\}
		\bigl(\widetilde{\mathfrak{A}}^{\rm hom}_{\theta}+A^{\rm hom}_{\varepsilon\theta}\bigr)^{-1} \widehat{J}(\theta) e_\theta(\cdot)d\theta,
	\end{align*}
	it remains to obtain a suitable smallness estimate for  
	\begin{align}\label{integral_theta}
		(2\pi)^{-3/2}A^{-1/2}\biggl(\frac{\cdot}{\varepsilon}\biggr)\int_{\R^3/\varepsilon^{-1}Q'} \biggl\{\overline{e}_{\varepsilon\theta}\nabla({e}_{\varepsilon\theta}\Psi_{\varepsilon\theta})\biggl(\frac{\cdot}{\varepsilon}\biggr)+I\biggr\}
		\bigl(\widetilde{\mathfrak{A}}^{\rm hom}_{\theta}+A^{\rm hom}_{\varepsilon\theta}\bigr)^{-1}\widehat{J}(\theta) e_\theta(\cdot)d\theta.
	\end{align}
	To this end, note that using the estimates \eqref{bound_A_hom} we have
	\begin{align}\label{sup_theta}
		\sup_{\theta \in \R^3/\varepsilon^{-1}Q'} \left| 
		\bigl(\widetilde{\mathfrak{A}}^{\rm hom}_{\theta}+A^{\rm hom}_{\varepsilon\theta}\bigr)^{-1} \right|
		\leq\bigl(\varepsilon^{-2}\pi^2\bigl\Vert(\widetilde{A}^{\rm hom})^{-1}\bigr\Vert+\nu_1\bigr)^{-1}\le\pi^{-2}\Vert(\widetilde{A}^{\rm hom})^{-1}\bigr\Vert^{-1}\varepsilon^2.
	\end{align}
	Using Parseval's identity and \eqref{sup_theta}, we can estimate the $L^2$ norm of \eqref{integral_theta} as follows:
	\begin{equation}
		\label{estimate_secondbit}
		\begin{aligned}
			&\bigg\| (2\pi)^{-3/2}A^{-1/2}\biggl(\frac{\cdot}{\varepsilon}\biggr)\int_{\R^3/\varepsilon^{-1}Q'}
			\biggl\{\overline{e}_{\varepsilon\theta}\nabla({e}_{\varepsilon\theta}\Psi_{\varepsilon\theta})\biggl(\frac{\cdot}{\varepsilon}\biggr)+I\biggr\}
			\bigl(\widetilde{\mathfrak{A}}^{\rm hom}_{\theta}+A^{\rm hom}_{\varepsilon\theta}\bigr)^{-1} 
			\widehat{J}(\theta) e_\theta(\cdot)d\theta \bigg\|_{L^2(\R^3,d\mu^\varepsilon)}\\[0.6em]
			&\leq (2\pi)^{-3/2}\max_{Q}\bigl\Vert A^{-1/2}\bigr\Vert
			\left(\big\| \nabla({e}_{\varepsilon\theta}\Psi_{\varepsilon\theta})\big\|_{L^2(Q,d\mu)}+I\right)\pi^{-2}\varepsilon^2\big\|\widehat{J}\big\|_{L^2(\R^3)}
			\leq 
			C\varepsilon^2\big\|J\big\|_{L^2(\R^3,d\mu^\varepsilon)}.
		\end{aligned}
	\end{equation}
	In the last inequality we use the fact that \eqref{eq_Psikappa}
	implies a uniform bound on $\|  \nabla({e}_{\varkappa}\Psi_{\varkappa})\|_{L^2(Q,d\mu)},$ $\varkappa\in{\mathbb R}^3.$ Combining \eqref{estimate_firstbit} and \eqref{estimate_secondbit} yields the claim. 
\end{proof}

\subsection{Asymptotic approximation of $\mathfrak{D}^\varepsilon_\theta$} 
\label{est_sec}
We now proceed to the proof of Theorem \ref{main_thm_gen}. For the rest of Section \ref{displacement_sec}, we use $d_\theta^\varepsilon$ as shorthand for $d_\theta^\varepsilon(G).$

For each $\theta \in \varepsilon^{-1} Q'$, $\varepsilon>0$, we write 
\begin{equation}
	\label{D=U+z_general}
{\mathfrak D}^\varepsilon_\theta:=A^{-1/2}\bigl(\overline{e}_{\varepsilon\theta} \nabla (e_{\varepsilon\theta} \Psi_{\varepsilon\theta})+I\bigr) d_\theta^\varepsilon+\varepsilon N({\rm i}\theta\times d_\theta^\varepsilon)+\varepsilon^2 R^\varepsilon_\theta
+z^\varepsilon_\theta.
\end{equation}
Here 
$N:=\widetilde{N}+A^{-1/2}a_\theta,$
where $a_\theta \in \R^{3\times 3}$ is such that $a_\theta\cdot\theta=0,$ $a_\theta\zeta\cdot\theta=0$ for all $\zeta\in{\mathbb R}^3,$ $\zeta\cdot\theta=0,$ 
and 
\begin{equation}
\int_Q{\rm i}\theta\times \widetilde{A}\bigl({\rm i}\theta \times A^{1/2} N({\rm i}\theta\times c)\bigr)=0\qquad \forall c\in{\mathbb R}^3.
\label{a_id}
\end{equation}
The existence of a unique $a_\theta$ 
is established in the same way as in \cite[Proposition 5.5, Lemma 5.6]{CDO22}.

For every $\phi\in H^1_{\curl A^{1/2},\varepsilon\theta}(Q, d\mu),$ consider the constant
\begin{equation}
\alpha_\phi:=\biggl(\int_QA^{-1}\biggr)^{-1}\int_QA^{-1/2}\phi.
\label{alphaphi}
\end{equation}
Notice that $\phi-A^{-1/2}\alpha_\phi\in H^1_{\curl A^{1/2},\varepsilon\theta}(Q, d\mu), $ 
and
\begin{equation}
\int_QA^{-1/2}\bigl(\phi-A^{-1/2}\alpha_\phi\bigr)=0.
\label{3stars}
\end{equation}
The term $R^\varepsilon_\theta \in H^1_{\curl A^{1/2},\varepsilon, \theta}(Q, d\mu)$ in \eqref{D=U+z_general} is defined as the solution to
\begin{equation}
\begin{aligned}
\overline{e}_{\varepsilon\theta}A^{1/2}&\curl\widetilde{A} \curl(e_{\varepsilon\theta} A^{1/2} R^\varepsilon_\theta)+\varepsilon^2\bigl(R_\theta^\varepsilon-(\widetilde{R_\theta^\varepsilon})_{\varepsilon\theta}\bigr)
\\[0.4em]
&=-A^{1/2} G-\varepsilon^{-1} A^{1/2}\overline{e}_{\varepsilon\theta}\curl\bigl(e_{\varepsilon\theta}\widetilde{A}({\rm i}\theta\times d^\varepsilon_\theta)\bigr)-\varepsilon^{-1} A^{1/2}\overline{e}_{\varepsilon\theta}\curl \widetilde{A} \curl\bigl(e_{\varepsilon\theta}A^{1/2} N({\rm i}\theta\times d_\theta^\varepsilon)\bigr)\\[0.4em]
&\ \ \ -A^{-1/2}\overline{e}_{\varepsilon\theta}\bigl(\nabla (e_{\varepsilon\theta}\Psi_{\varepsilon\theta}) +I\bigr) d^\varepsilon_\theta =:\mathfrak{H}^\varepsilon_\theta \in\bigl(H^1_{\curl A^{1/2},\varepsilon\theta}(Q, d\mu)\bigr)^*.
\end{aligned}
\label{R_eq_general}
\end{equation}
This is understood in the sense of the integral identity
\begin{equation}
\label{R_eq_weak_general}
\begin{aligned}
&\int_Q \widetilde{A}\curl(e_{\varepsilon\theta} A^{1/2} R^\varepsilon_\theta)\cdot\overline{\curl(e_{\varepsilon\theta} A^{1/2} \varphi)}+\varepsilon^2 \int_Q\bigl(R_\theta^\varepsilon-(\widetilde{R_\theta^\varepsilon})_{\varepsilon\theta}\bigr)\cdot\overline{\varphi} 
\\[0.4em]
&
=-\int_Q\Bigl( A^{1/2} G+A^{-1/2}\bigl(\overline{e}_{\varepsilon\theta}\nabla(e_{\varepsilon\theta}\Psi_{\varepsilon\theta})+I\bigr)d_\theta^\varepsilon \Bigr) 
\cdot\overline{\varphi}\\[0.4em]
&\hspace{5em}-\varepsilon^{-1}\int_Qe_{\varepsilon\theta}  \widetilde{A}\bigl\{({\rm i}\theta\times d^\varepsilon_\theta)+\overline{e}_{\varepsilon\theta}\curl\bigl(e_{\varepsilon\theta}A^{1/2} N({\rm i}\theta\times d_\theta^\varepsilon)\bigr)\bigr\}\cdot\overline{\curl\bigl(e_{\varepsilon\theta} A^{1/2}\varphi\bigr)}
\\[0.4em]
&
=-\int_Q\Bigl( A^{1/2} G+A^{-1/2}\bigl(\overline{e}_{\varepsilon\theta}\nabla(e_{\varepsilon\theta}\Psi_{\varepsilon\theta})+I\bigr)d_\theta^\varepsilon \Bigr) 
\cdot\overline{\varphi}\\[0.4em]
&\hspace{5em}-\varepsilon^{-1}\int_Qe_{\varepsilon\theta}  \widetilde{A}\bigl\{({\rm i}\theta\times d^\varepsilon_\theta)+\curl\bigl(A^{1/2} N({\rm i}\theta\times d_\theta^\varepsilon)\bigr)\bigr\}\cdot\overline{\curl\bigl(e_{\varepsilon\theta} A^{1/2}(\varphi-A^{-1/2}\alpha_\varphi)\bigr)}
\\[0.4em]
&\hspace{5em}-\varepsilon^{-1}\int_Q \widetilde{A}\bigl\{({\rm i}\theta\times d^\varepsilon_\theta)+\curl\bigl(A^{1/2} N({\rm i}\theta\times d_\theta^\varepsilon)\bigr)\bigr\}\cdot\overline{({\rm i}\varepsilon\theta\times\alpha_\varphi)}
\\[0.4em]
&
\hspace{5em}
-\int_Qe_{\varepsilon\theta}\widetilde{A}\bigl({\rm i}\theta\times A^{1/2} N({\rm i}\theta\times d_\theta^\varepsilon)\bigr)
\cdot\overline{\curl (e_{\varepsilon\theta} A^{1/2} \varphi)}
\\[0.4em]
&
=-\int_Q\Bigl( A^{1/2} G+A^{-1/2}\bigl(\overline{e}_{\varepsilon\theta}\nabla(e_{\varepsilon\theta}\Psi_{\varepsilon\theta})+I\bigr)d_\theta^\varepsilon \Bigr) 
\cdot\overline{\varphi}
-\int_Q \widetilde{A}^{\rm hom}({\rm i}\theta\times d_\theta^\varepsilon)\cdot\overline{({\rm i}\theta\times\alpha_\varphi)}
\\[0.4em]
&
\hspace{5em}-\int_Qe_{\varepsilon\theta}\widetilde{A}\bigl({\rm i}\theta\times A^{1/2} N({\rm i}\theta\times d_\theta^\varepsilon)\bigr)
\cdot\overline{\curl (e_{\varepsilon\theta} A^{1/2} \varphi)}
\equiv\langle \mathfrak{H}^\varepsilon_\theta,\varphi\rangle\quad \forall\varphi\in H^1_{\curl A^{1/2},\varepsilon\theta}(Q, d\mu),
\end{aligned}
\end{equation}
where for the second equality we used the fact that 
\begin{equation*}
\overline{e}_{\varepsilon\theta}\curl\bigl(e_{\varepsilon\theta}A^{1/2} N({\rm i}\theta\times d_\theta^\varepsilon)\bigr)={\rm i}\theta\times A^{1/2} N({\rm i}\theta\times d_\theta^\varepsilon)+\curl\bigl(A^{1/2} N({\rm i}\theta\times d_\theta^\varepsilon)\bigr)
\end{equation*}
and for the third equality we invoke the definition of $\widetilde{A}^{\rm hom},$ see \eqref{Ahom_tilde_def}, as well as the fact that the function $N$ satisfies the integral identity (cf. \eqref{cell_pbN_general_weak})
\[
\int_Q \widetilde{A} \curl (A^{1/2}N)\cdot\overline{\curl (A^{1/2}\psi)}d\mu=-\int_Q \widetilde{A}\,\overline{\curl (A^{1/2}\psi)} \qquad \forall \psi\in H^1_{\curl A^{1/2}}(Q, d\mu),\ \  \int_Q A^{-1/2}\psi=0,
\]
setting (cf. (\ref{3stars})) $\psi=\varphi-A^{-1/2}\alpha_\varphi.$

\begin{prop}\label{prop_existenceR1}
	For each $\varepsilon>0,$ $\theta\in{\mathbb R}^3,$ there exists a unique solution $R^\varepsilon_\theta \in H^1_{\curl A^{1/2},\varepsilon\theta}(Q, d\mu)$ to \eqref{R_eq_general}.
\end{prop}
\begin{proof}
	
	
	Writing the decomposition \eqref{decokneq0} for the test function $\varphi$ in \eqref{R_eq_weak_general} and using the fact that $R_\theta^\varepsilon-(\widetilde{R_\theta^\varepsilon})_{\varepsilon\theta}$ and $\widetilde{\psi}_{\varepsilon\theta}$ are orthogonal in $L^2(Q,d\mu),$ one has
	\[
	\int_Q\bigl(R_\theta^\varepsilon-(\widetilde{R_\theta^\varepsilon})_{\varepsilon\theta}\bigr)\cdot\overline{\psi}=\int_Q\bigl(R_\theta^\varepsilon-(\widetilde{R_\theta^\varepsilon})_{\varepsilon\theta}\bigr)\cdot\overline{\bigl(\psi-\widetilde{\psi}_{\varepsilon\theta}\bigr)}.
	\]
	The claim now follows from the Lax-Milgram lemma applied to the sesquilinear form
	\begin{align*}
		\int_Q \curl(e_{\varepsilon\theta} A^{1/2} u) \cdot\overline{\curl(e_{\varepsilon\theta} A^{1/2} v)}
		+\varepsilon^2 \int_Q(u-\widetilde{u}_{\varepsilon\theta})\cdot\overline{(v-\widetilde{v}_{\varepsilon\theta})}, 
		\qquad u,v \in H^1_{\curl A^{1/2},\varepsilon\theta}(Q,d\mu),
	\end{align*}
	which is clearly bounded and is also coercive due to the
	the Poincar\'{e}-type inequality \eqref{poincare_ineq_A}.
\end{proof}

We next establish a property of $\mathfrak{H}^\varepsilon_\theta$ that we use in Section \ref{R_est_sec} for estimating the remainder $\varepsilon^2R^\varepsilon_\theta$ in \eqref{D=U+z_general}. 

\begin{prop}
	\label{tildethm1} 
	For every $w\in H^1_{\curl A^{1/2},\varepsilon\theta}(Q, d\mu)$ one has
		$\langle \mathfrak{H}^\varepsilon_\theta, w\rangle=\langle \mathfrak{H}^\varepsilon_\theta, \widetilde{w}_{\varepsilon\theta}\rangle,$
	where $\widetilde{w}_{\varepsilon\theta}$ is the first term in the decomposition \eqref{decokneq0}.
\end{prop}
\begin{proof}
	Referring to the definition of $\mathfrak{H}^\varepsilon_\theta$ in \eqref{R_eq_general}, we have, for all $\varphi\in H^1_{\#,0},$ 
	\begin{equation*}
		\begin{aligned}
			\bigl\langle \mathfrak{H}^\varepsilon_\theta, A^{-1/2}\overline{e}_{\varepsilon\theta}\nabla(e_{\varepsilon\theta} \varphi)\bigr\rangle=&
			\int_Q e_{\varepsilon\theta} G\cdot\overline{\nabla(e_{\varepsilon\theta} \varphi)}+\int_Q A^{-1}\bigl(\overline{e}_{\varepsilon\theta}\nabla(e_{\varepsilon\theta} \Psi_{\varepsilon\theta})+I\bigr) d_\theta^\varepsilon\cdot{{e}_{\varepsilon\theta}\overline{\nabla(e_{\varepsilon\theta}\varphi)}}=0,
		\end{aligned}
	\end{equation*}
	where both integral integrals vanish, since  $\overline{e}_{\varepsilon\theta}\divv (e_{\varepsilon\theta} G)=0$ 
	and
	$\Psi_{\varepsilon\theta}$ solves \eqref{eq_Psikappa} with $\varkappa=\varepsilon\theta.$ 

	Furthermore, for all $c\in{\mathbb R}^3$  one has (see \eqref{alphaphi}) $\alpha_{A^{-1/2}c}=c,$
	and hence
	\begin{align*}
		\bigl\langle \mathfrak{H}^\varepsilon_\theta,  A^{-1/2}c\bigr\rangle&=
		-\int_Q G\cdot c-\int_Q A^{-1}\bigl(\overline{e}_{\varepsilon\theta}\nabla(e_{\varepsilon\theta} \Psi_{\varepsilon\theta})+I\bigr)d_\theta^\varepsilon \cdot c
		-\int_Q{\rm i}\theta\times\widetilde{A}^{\rm hom}({\rm i}\theta\times  d_\theta^\varepsilon)\cdot c
		\\[0.4em]
		&
		=-\int_Q
		\bigl(G+
		\bigl(\widetilde{\mathfrak A}^{\rm hom}_\theta+{A}^{\rm hom}_{\varepsilon\theta}\bigr)
		d_\theta^\varepsilon\bigr)\cdot c=0,
	\end{align*}
	where the last equality follows from \eqref{dtheta_eq_gen}.
\end{proof}

	\subsection{Estimate for $R^\varepsilon_\theta$}

\label{R_est_sec}

\begin{thm}
\label{thm_Restimate_gen}
There exists $C>0$ such that for all $\varepsilon>0$ and $\theta\in \varepsilon^{-1}Q'$, the solution $R^\varepsilon_\theta$ of \eqref{R_eq_general} satisfies
\begin{align}
&\bigl\|(\widetilde{R^\varepsilon_\theta})_{\varepsilon\theta}
\bigr\|_{L^2(Q, d\mu)}\leq C \|G\|_{L^2(Q, d\mu)},\qquad\qquad
\bigl\|R^\varepsilon_\theta-(\widetilde{R^\varepsilon_\theta})_{\varepsilon\theta} 
\bigr\|_{L^2(Q, d\mu)}\leq C\varepsilon^{-1} \|G\|_{L^2(Q, d\mu)}.
\label{R2_A_gen}
\end{align}
\end{thm}

\begin{proof}
	Setting $\varphi=R^\varepsilon_\theta$ in (\ref{R_eq_weak_general}), we obtain
	\begin{equation*}
		\begin{aligned}
			\int_Q\widetilde{A}\curl\bigl(e_{\varepsilon\theta} A^{1/2} R^\varepsilon_\theta\bigr)\cdot\overline{\curl(e_{\varepsilon\theta} A^{1/2}R^\varepsilon_\theta)}+\varepsilon^2 \int_Q\bigl(R^\varepsilon_\theta-(\widetilde{R^\varepsilon_\theta})_{\varepsilon\theta}\bigr)\cdot\overline{R^\varepsilon_\theta}
			=\bigl\langle\mathfrak{H}^\varepsilon_\theta, R^\varepsilon_\theta\bigr\rangle.
		\end{aligned}
	\end{equation*}
	Taking into account the $L^2$-orthogonality of $R^\varepsilon_\theta-(\widetilde{R^\varepsilon_\theta})_{\varepsilon\theta}$ and $(\widetilde{R^\varepsilon_\theta})_{\varepsilon\theta}$ as well as Proposition \ref{tildethm1} yields
	\begin{equation}
		\begin{aligned}
			\int_Q \widetilde{A}\curl(e_{\varepsilon\theta} A^{1/2} R^\varepsilon_\theta) \cdot\overline{\curl(e_{\varepsilon\theta} A^{1/2} R^\varepsilon_\theta)}&+\varepsilon^2 \int_Q\bigl|R^\varepsilon_\theta-(\widetilde{R^\varepsilon_\theta})_{\varepsilon\theta}\bigr|^2
			=\bigl\langle\mathfrak{H}^\varepsilon_\theta,(\widetilde{R^\varepsilon_\theta})_{\varepsilon\theta}\bigr\rangle.
		\end{aligned}
		\label{R_quadratic_noRtilde_general}
	\end{equation}
Noting that $\int_{Q}A^{-1/2}\widetilde{(R^\varepsilon_\theta)}_{\varepsilon_\theta}=0$ (see \eqref{mean_wtilde}), and hence $\alpha_\varphi=0$ for $\varphi={(\widetilde{R^\varepsilon_\theta})_{\varepsilon\theta}}$ (see \eqref{alphaphi}), we write
\begin{equation*}
	\begin{aligned}
\bigl\langle\mathfrak{H}^\varepsilon_\theta,(\widetilde{R^\varepsilon_\theta})_{\varepsilon\theta}\bigr\rangle=&-\int_Q\Bigl( A^{1/2} G+A^{-1/2}\bigl(\overline{e}_{\varepsilon\theta}\nabla(e_{\varepsilon\theta}\Psi_{\varepsilon\theta})+I\bigr)d_\theta^\varepsilon \Bigr) 
\cdot\overline{(\widetilde{R^\varepsilon_\theta})_{\varepsilon\theta}}
\\[0.4em]
&-\int_Qe_{\varepsilon\theta}\widetilde{A}\bigl({\rm i}\theta\times A^{1/2} N({\rm i}\theta\times d_\theta^\varepsilon)\bigr)
\cdot\overline{\curl\bigl(e_{\varepsilon\theta} A^{1/2}R^\varepsilon_\theta\bigr)},
\end{aligned}
\end{equation*}
where we have used (\ref{a_id}).

Now, combining (\ref{R_quadratic_noRtilde_general}) with the Poincar\'{e} inequality provided by Theorem \ref{thm_curlpoincare}, 
we obtain the first estimate in (\ref{R2_A_gen}).
Finally, using (\ref{R_quadratic_noRtilde_general}) once again immediately yields the second estimate in 
\eqref{R2_A_gen}.
\end{proof}



\begin{cor}\label{corollary_U_general}
	There is a constant $C>0$ such that for all $\varepsilon$, $\theta,$ $G$ one has
		$\bigl\|R^\varepsilon_\theta 
		\bigr\|_{L^2(Q,d\mu)}
		\leq C \varepsilon^{-1}\|G\|_{L^2(Q,d\mu)}.$
\end{cor}

\begin{rmk}
	The first estimate in (\ref{R2_A_gen}) and equation (\ref{R_quadratic_noRtilde_general}) yield 
	\begin{equation}
		\bigl\|\curl(e_{\varepsilon\theta} A^{1/2} R^\varepsilon_\theta)\bigr\|_{L^2(Q,d\mu)}\le C\|G\|_{L^2(Q,d\mu)}.
		\label{Rcurl_general}
	\end{equation}
\end{rmk}

\subsection{Conclusion of the proof of Theorem \ref{main_thm_gen}}

\label{z_est_sec}

\begin{prop}\label{prop_z_general}
There exists $C>0$ such that the remainder $z^\varepsilon_\theta$ defined in \eqref{D=U+z_general} satisfies the estimate
\begin{equation}\label{z_estimate_gen}
\|z^\varepsilon_\theta\|_{L^2(Q,d\mu)}\leq C \varepsilon \|G\|_{L^2(Q, d\mu)},\qquad \bigl\Vert\curl\bigl(e_{\varepsilon\theta} A^{1/2} z^\varepsilon_\theta\bigr)\bigr\Vert_{L^2(Q, d\mu)}\le C\varepsilon^2\Vert G\Vert_{L^2(Q, d\mu)}.
\end{equation}
\end{prop}
\begin{proof}
The function $z^\varepsilon_\theta \in H^1_{\curl A^{1/2},\varepsilon\theta}(Q, d\mu)$ solves
\begin{align}\label{z_eq_gen}
\varepsilon^{-2} A^{1/2}\overline{e}_{\varepsilon\theta}\curl \widetilde{A}\curl\bigl(e_{\varepsilon\theta} A^{1/2} z^\varepsilon_\theta\bigr) + z^\varepsilon_\theta =-\varepsilon N({\rm i}\theta\times d_\theta^\varepsilon)-\varepsilon^2 (\widetilde{R^\varepsilon_\theta})_{\varepsilon\theta},
\end{align}
understood im the sense that 
\begin{equation}
	\begin{aligned}
		\varepsilon^{-2}\int_Q\curl(e_{\varepsilon\theta} A^{1/2}z^\varepsilon_\theta)&\cdot\overline{\curl(e_{\varepsilon\theta} A^{1/2}\varphi)}+\int_Qz^\varepsilon_\theta\cdot\overline{\varphi}
		\\[0.4em]
		&=-\varepsilon^2\int_Q (\widetilde{R^\varepsilon_\theta})_{\varepsilon\theta}\cdot\overline{\varphi}-\varepsilon\int_QN({\rm i}\theta\times d_\theta^\varepsilon)\cdot\overline{\varphi}
		\qquad \forall\varphi\in H^1_{\curl A^{1/2},\varepsilon\theta}(Q, d\mu).
		\label{z_eq_var_general}
	\end{aligned}
\end{equation}
Setting $\varphi=z^\varepsilon_\theta$ in (\ref{z_eq_var_general}), we obtain
\begin{equation}
\varepsilon^{-2}\int_Q \widetilde{A}  \curl(e_{\varepsilon\theta} A^{1/2}z^\varepsilon_\theta)\cdot\overline{\curl(e_{\varepsilon\theta} A^{1/2} z^\varepsilon_\theta)}+\int_Q |z^\varepsilon_\theta|^2=-\varepsilon\int_Q({\rm i}\theta\times d_\theta^\varepsilon) \cdot\overline{z^\varepsilon_\theta}-\varepsilon^2\int_Q (\widetilde{R^\varepsilon_\theta})_{\varepsilon\theta}\cdot\overline{z^\varepsilon_\theta}.
\label{641}
\end{equation}
	Applying the H\"{o}lder inequality to both terms on the right-hand side of \eqref{641}, using Theorem \ref{thm_Restimate_gen} to estimate $\widetilde{R}^\theta_\varepsilon$ and recalling the definition \eqref{dtheta_eq_gen} of $d_\theta^\varepsilon$ yields the first estimate in \eqref{z_estimate_gen}.
Using this to estimate the right-hand side of \eqref{641} once again, in conjunction with the first term on its left-hand side, yields the second estimate in \eqref{z_estimate_gen}.
\end{proof}


Proposition \ref{prop_z_general} and Corollary \ref{corollary_U_general} imply \eqref{main_estimate_gen}, since
\begin{equation*}
\bigl\| {\mathfrak D}^\varepsilon_\theta -A^{-1/2}\bigl(\overline{e}_{\varepsilon\theta}\nabla({e}_{\varepsilon\theta}\Psi_{\varepsilon\theta})+I\bigr){d}_\theta^\varepsilon\bigr\|_{L^2(Q, d\mu)}
\leq \varepsilon^2\bigl\|R^\varepsilon_\theta 
\bigr\|_{L^2(Q, d\mu)}+ \| z^\varepsilon_\theta\|_{L^2(Q, d\mu)},
\end{equation*}
which concludes the proof of Theorem \ref{main_thm_gen}.

\section{Magnetic induction $B^\varepsilon$} 

To obtain the estimates for the magnetic field and induction, we recall that the Floquet transform of \eqref{Max_general} is
\begin{equation}
	\label{Max_ft_BD_gen}
\varepsilon^{-1}\overline{e}_{\varepsilon\theta}\curl(e_{\varepsilon\theta} A^{1/2} {\mathfrak D}^\varepsilon_\theta)+B^\varepsilon_\theta=0,\qquad\quad
\varepsilon^{-1} A^{1/2} \overline{e}_{\varepsilon\theta}\curl(e_{\varepsilon\theta} \widetilde{A} B^\varepsilon_\theta)-{\mathfrak D}^\varepsilon_\theta =A^{1/2}J^\varepsilon_\theta,
\end{equation}
where ${J}_\theta^\varepsilon:=\overline{e}_{\varepsilon\theta}\mathcal{F}_\varepsilon J$ (and, in particular, $\overline{e}_{\varepsilon\theta}\divv (e_{\varepsilon\theta} A^{-1/2}J^\varepsilon_\theta)=0$) and $B^\varepsilon_\theta := \overline{e}_{\varepsilon\theta}\mathcal{F}_\varepsilon B^\varepsilon$ is the Floquet-transformed magnetic induction.

To find an asymptotic approximation for $B^\varepsilon_\theta,$ we use the expression \eqref{D=U+z_general} for ${\mathfrak D}^\varepsilon_\theta$ and the first equation of \eqref{Max_ft_BD_gen}, which yields
\begin{align*}
B^\varepsilon_\theta &= \varepsilon^{-1}\overline{e}_{\varepsilon\theta}\curl\Bigl\{e_{\varepsilon\theta}\big( (\overline{e}_{\varepsilon\theta} \nabla(e_{\varepsilon\theta} \Psi_{\varepsilon\theta}) +I)d_\theta^\varepsilon + \varepsilon A^{1/2}N({\rm i}\theta\times d^\varepsilon_\theta)+\varepsilon^2 A^{1/2}R^\varepsilon_\theta+A^{1/2}z^\varepsilon_\theta\big)\Bigr\}\\[0.3em]
&
=\bigl(\curl(A^{1/2}\widetilde{N})+I\bigr)({\rm i}\theta\times d_\theta^\varepsilon) +\varepsilon \big\{{\rm i}\theta \times\bigl(A^{1/2} N({\rm i}\theta\times d^\varepsilon_\theta)\bigr)\\[0.3em]
&\hspace{4cm}+\overline{e}_{\varepsilon\theta}\curl (e_{\varepsilon\theta}A^{1/2}R^\varepsilon_\theta)\big\}+\varepsilon^{-1}\overline{e}_{\varepsilon\theta}\curl(e_{\varepsilon\theta} A^{1/2}z^\varepsilon_\theta).
\end{align*}
Here $d_\theta^\varepsilon=d_\theta^\varepsilon(J^\varepsilon_\theta)$ is defined in \eqref{dtheta_eq_gen}, $\widetilde{N}$ is defined by (\ref{cell_pbN_general}), $R^\varepsilon_\theta$ solves \eqref{R_eq_general} with $G=J^\varepsilon_\theta$, and $z^\varepsilon_\theta$ solves (\ref{z_eq_gen}). As a consequence of (\ref{Rcurl_general}) and the second estimate in (\ref{z_estimate_gen}), the following result holds.
\begin{thm}
There exists $C>0$ independent of $J,$ $\theta$, $\varepsilon$ such that 
\begin{align}
\bigl\| B^\varepsilon_\theta - \bigl(\curl(A^{1/2}\widetilde{N}) +I\bigr)\bigl({\rm i}\theta\times d_\theta^\varepsilon(J^\varepsilon_\theta)\bigr)\bigr\|_{L^2(Q, d\mu)}\leq C\varepsilon\|J^\varepsilon_\theta\|_{L^2(Q,d\mu)},
\label{main_est_B_gen}
\end{align}
\end{thm}
Applying the inverse Floquet transform to \eqref{main_est_B_gen}, we obtain the following operator-norm estimate for the magnetic induction component $B^\varepsilon$ of the solution of \eqref{Max_general}. 
\begin{cor}
There exists a constant $C>0$ such that, for all $\varepsilon>0$ and divergence-free $J\in L^2({\mathbb R}^3, d\mu^\varepsilon),$ 
\begin{equation}
	\begin{aligned}
\biggl\| B^\varepsilon-(2\pi)^{-3/2} \biggl\{\curl \bigl(A^{1/2}\widetilde{N}\bigr)\biggl(\frac{\cdot}{\varepsilon}\biggr)+I\biggr\}\curl\int_{\R^3}  
 \big(\mathfrak{\widetilde{A}}_\theta^{\rm hom}&+{A}^{\rm hom}_{\varepsilon\theta}\big)^{-1} \widehat{J}(\theta) e_\theta  d\theta\biggr\|_{L^2(\R^3, d\mu^\varepsilon)}\\[0.4em]
 &\leq C \varepsilon \|J\|_{L^2(\R^3, d\mu^\varepsilon)},
 \end{aligned}
\label{B_est_final}
\end{equation}
Here $\mathfrak{\widetilde{A}}_\theta^{\rm hom},$ $\theta\in{\mathbb R}^3,$ and $A^{\rm hom}_\varkappa,$ $\varkappa\in[-\pi,\pi)^3,$ are defined by \eqref{Atheta_tilde} and \eqref{A_hom_hat}, respectively.
\end{cor}

\subsection{Comparison to the formal expansion approach}
\label{comparison_sec}

Suppose that $\widetilde{A}=1$ (i.e. the magnetic permeability is that of vacuum) and $\mu$ is the Lebesgue measure on $Q.$ Then, according to \eqref{cell_pbN_general},
one has $\widetilde{N}=0,$ $\widetilde{A}^{\rm hom}=I,$ and so (see \eqref{Atheta_tilde}) 
$(\widetilde{\mathfrak A}^{\rm hom}_\theta)_{ij}=\varepsilon_{isk}({\rm i}\theta)_s\varepsilon_{ktj}({\rm i}\theta)_t,$ $i,j=1,2,3.$
The estimate \eqref{B_est_final} takes the form 
\begin{equation}
		\biggl\| B^\varepsilon-(2\pi)^{-3/2}\curl\int_{\R^3}  
		\big(\widetilde{\mathfrak{A}}^{\rm hom}_\theta+A^{\rm hom}_{\varepsilon\theta}\big)^{-1}\widehat{J}(\theta) e_\theta  d\theta\biggr\|_{L^2(\R^3, d\mu^\varepsilon)}
		\leq C \varepsilon \|J\|_{L^2(\R^3, d\mu^\varepsilon)},
\label{B_est_final1}
\end{equation}

Note that \eqref{B_est_final1} involves a pseudo-differential operator with a ``two-scale" symbol that depends on $\theta$ and $\varepsilon\theta$.  
Formally setting $\varepsilon=0$ in \eqref{B_est_final1} yields
\begin{equation}
	(2\pi)^{-3/2}\curl\int_{\R^3}\bigl(\widetilde{\mathfrak A}_\theta^{\rm hom}+A^{\rm hom}_0\bigr)^{-1}\widehat{J}(\theta) e_\theta(\cdot)d\theta
	=\curl\bigl(\curl\curl+(A^{\rm hom})^{-1}\bigr)^{-1}
	J,
	\label{Bapprox_reduced}
\end{equation}
where $A^{\rm hom}$ is the standard homogenised inverse electric permittivity, see (\ref{Max_general_homo}). The ``homogenised" equation for the Maxwell system in the whole space \eqref{Max_general_homo} obtained by two-scale expansion (where $\widetilde{A}^{\rm hom}=I$), we have 
$\big(\!\curl\curl+(A^{\rm hom})^{-1}\big)^{-1}J=-A^{\rm hom}D^{\rm hom}.$
Hence, the formal approximation $B^{\rm hom}$ 
obtained from \eqref{Max_general_homo} takes the form
\[
B^{\rm hom}=-\curl(A^{\rm hom}
{D}^{\rm hom})=\curl\big(\!\curl\curl+(A^{\rm hom})^{-1}\big)^{-1}J,
\]
which coincides with (\ref{Bapprox_reduced}).

In other words, the approximation for magnetic induction $B^\varepsilon$ obtained formally by setting $\varepsilon=0$ in \eqref{B_est_final} coincides with that obtained na\"{\i}vely by using the classical two-scale asymptotic ansatz.


\renewcommand{\theequation}{A.\arabic{equation}}
\renewcommand{\thesubsection}{A.\arabic{subsection}}
\renewcommand{\thethm}{B.\arabic{thm}}
\setcounter{equation}{0}
\setcounter{equation}{0}
\setcounter{thm}{0}

\section*{Appendix A: Formal two-scale asymptotics}

	 Suppose $\widetilde{A}=1,$ as in Section \ref{comparison_sec}. Then the equation for the magnetic induction $B^\varepsilon$ (see \eqref{Max_general}) takes the form
	\begin{equation}
		\label{MaxII}
		\curl\bigl(A(\cdot/\varepsilon)\curl B^\varepsilon\bigr) +B^\varepsilon=\curl\bigl(A(\cdot/\varepsilon)J\bigr)
	\end{equation}
	
	The standard approach is to seek the two-scale expansion 
	\[
	B^\varepsilon=\sum_{n=0}^\infty\varepsilon^nB_n\biggl(x, \frac{x}{\varepsilon}\biggr),\qquad B_n(x, y) \text{\ is\  periodic\ in\  }y,\quad n=0,1,2,\dots.
	\]
	
	Formally substituting into (\ref{MaxII}) and applying $\curl=\curl_x+\varepsilon^{-1}\curl_y$ yields, after rearranging the terms by similar powers of $\varepsilon,$ to the leading order: 
	\[
	\curl_y\bigl(A(y)\curl_yB_0(x, y)\bigr)=0,\quad \divv_y B_0(x, y)=0,
	\]
	hence $B_0$ is independent of $y,$ and one can write $B_0=B_0(x).$ To the next order in $\varepsilon,$ one obtains
	\begin{align*}
		&\curl_y\bigl(A(y)\curl_yB_1(x, y)\bigr)=-\curl_y \bigl(A(y)\curl_xB_0(x)\bigr)+\curl_y\bigl(A(y)J(x)\bigr),\\[0.3em]
		&\divv_y B_1(x, y)+\divv B_0(x)=0,
	\end{align*}
	and the condition of periodicity in $y$ implies
	$\divv B_0=0,$ hence $\divv_y B_1(x, y)=0.$ Solving for $B_1$ yields
	\[
	B_1(x, y)=N(y)\bigl(\curl B_0(x)-J(x)\bigr),
	\]
	where the matrix-valued function  $N$ solves
	\[
	\curl\bigl\{A\bigl(\curl N+I\bigr)\bigr\}=0,\qquad \divv N=0,\qquad \int_Q N=0.
	\]
	
Finally, the solvability condition for
	\begin{align*}
		\curl_y\bigl(A(y)\curl_yB_2(x, y)\bigr)=&-\curl_x \bigl(A(y)\curl_yB_1(x,y)\bigr)-\curl_y \bigl(A(y)\curl_xB_1(x,y)\bigr)\\[0.3em]
		&-\curl_x \bigl(A(y)\curl_xB_0(x)\bigr)-B_0(x)+\curl_x\bigl(A(y)J(x)\bigr)
	\end{align*}
	results in the equation 
	$\curl\bigl(A^{\rm hom}\curl B_0\bigr)+B_0=\curl\bigl(A^{\rm hom}J\bigr),$
	where 
	\[
	A^{\rm hom}:=\int_QA\bigl(\curl N+I\bigr)d\mu,
	\]
	are ``homogenised'' coefficients.

\section*{Acknowledgments}
KC is grateful for the support of
the Engineering and Physical Sciences Research Council (UK), under the grant EP/L018802/2 ``Mathematical foundations of metamaterials: homogenisation, dissipation and operator theory''.

\end{document}